\documentclass[a4paper,12pt]{amsart}
\usepackage{doc}
\usepackage{amsmath,amsfonts,amssymb}
\usepackage{tikz}
\usepackage{multirow}
\usepackage{graphicx}
\usepackage{verbatim}
\usepackage{url}
\usepackage{hyperref}
\usepackage{float}
\floatstyle{plaintop}
\restylefloat{table}
\calclayout
\theoremstyle{plain}
\newtheorem{theorem}{Theorem}[section]
\newtheorem{lemma}[theorem]{Lemma}
\newtheorem{corollary}[theorem]{Corollary}
\newtheorem{proposition}[theorem]{Proposition}

\theoremstyle{definition}
\newtheorem*{definition}{Definition}

\theoremstyle{remark}

\newcommand{\Z}{\mathbb{Z}}

\newcommand{\Hy}{\mathbb{H}}

\newcommand{\Cb}{\mathbb{C}}

\begin{document}

\markboth{Ji-Young Ham, Joongul Lee, Alexander Mednykh and Aleksey Rasskazov}{Volumes of link $7_3^2$}


\title{An explicit volume formula for the link $7_3^2 (\alpha, \alpha)$ cone-manifolds
}

\author{Ji-Young Ham 
}

\address{Department of Science, Hongik University, 
94 Wausan-ro, Mapo-gu, Seoul,
 04066\\
  Korea.}
\email{jiyoungham1@gmail.com}

\author{Joongul Lee
}

\address{Department of Mathematics Education, Hongik University, 
94 Wausan-ro, Mapo-gu, Seoul,
04066\\
   Korea.}
\email{jglee@hongik.ac.kr}
   
\author{Alexander Mednykh*
}   
\thanks{*The author  was funded by the Russian Science Foundation (grant 16-41-02006).}
   
\address{Sobolev Institute  of Mathematics,  pr. Kotyuga 4, Novosibirsk  630090  \\
Novosibirsk State University,  Pirogova 2, Novosibirsk  630090\\
 Russia.}
\email{mednykh@math.nsc.ru}

\author{Aleksey Rasskazov
}

\address{Webster International University \\
146 Moo 5, Tambon Sam Phraya, Cha-am, Phetchaburi 76120\\
Tailand}
\email{arasskazov69@webster.edu}   

\maketitle

\begin{abstract}
 We calculate the volume of the $7_3^2$ link cone-manifolds using the Schl\"{a}fli formula.  As an application, we give the volume of the cyclic coverings branched over the link. 
\end{abstract}



\section{Introduction}

Let us denote the link complement of $7_3^2$ in Rolfsen's link table by $X$. Note that it is a hyperbolic knot. Hence by Mostow-Prasad rigidity theorem, $X$ has a unique hyperbolic structure. 
Let $\rho_{\infty}$ be the holonomy representation from $\pi_1(X)$ to ${\textnormal{PSL}}(2, \mathbb C)$ and denote $\rho_{\infty} \left(\pi_1(X)\right)$ by $\Gamma$, a Kleinian group.
$X$ is a $({\textnormal{PSL}}(2, \mathbb C),\Hy^3)$-manifold and can be identified with 
$\Hy^3/ \Gamma$. 
Thurston's orbifold theorem guarantees an orbifold, $X(\alpha)=X(\alpha,\alpha)$, with underlying space $S^3$ and with the  link $7_3^2$ as the singular locus of the cone-angle 
$\alpha=2 \pi/k$ for some nonzero integer $k$, can be identified with 
$\Hy^3/\Gamma^{\prime}$ for some $\Gamma^{\prime} \in {\textnormal{PSL}}(2, \mathbb C)$; the hyperbolic structure of $X$ is deformed to the hyperbolic structure of 
$X(\alpha)$.
For the intermediate angles whose multiples are not $2 \pi$ and not bigger than $\pi$, Kojima~\cite{K1} showed that the hyperbolic structure of 
$X(\alpha)$ can be obtained uniquely by deforming nearby orbifold structures.
Note that there exists an angle $\alpha_0 \in [\frac{2\pi}{3},\pi)$ for the link $7_3^2$ such that $X(\alpha)$ is hyperbolic for $\alpha \in (0, \alpha_0)$, Euclidean for $\alpha=\alpha_0$, and spherical for $\alpha \in (\alpha_0, \pi]$ \cite{P2,HLM1,K1,PW}. 
For further knowledge of cone-manifolds a reader can consult~\cite{CHK,HMP}.

Even though we have wide discussions on orbifolds, it seems to us we have a little in regard to cone-manifolds.
Explicit volume formulae for hyperbolic cone-manifolds of knots and links are known a little. The volume formulae for hyperbolic cone-manifolds of the knot 
$4_1$~\cite{HLM1,K1,K2,MR1}, the knot $5_2$~\cite{M2}, the link $5_1^2$~\cite{MV1}, 
the link $6_2^2$~\cite{M1}, and the link $6_3^2$~\cite{DMM1} have been computed. In~\cite{HLM2} a method of calculating the volumes of two-bridge knot cone-manifolds was introduced but without explicit formulae. In~\cite{HMP,HL1}, explicit volume formulae of cone-manifolds for the hyperbolic twist knot and for the knot with Conway notation $C(2n,3)$ are computed. Similar methods are used for computing Chern-Simons invariants of orbifolds for the twist knot  and $C(2n,3)$ knot in~\cite{HL,HL2}.

The main purpose of the paper is to find an explicit and efficient volume formula of hyperbolic cone-manifolds for the link $7_3^2$. The following theorem gives the volume formula for 
$X(\alpha)$.

\begin{theorem}\label{thm:main}
Let $X(\alpha)$, $0 \leq \alpha < \alpha_0$ be the hyperbolic cone-manifold with underlying space $S^3$ and with singular set the link $7_3^2$ of cone-angle 
$\alpha$. $X(0)$ denotes $X$. Then the volume of $X(\alpha)$ is given by the following formula

\begin{align*}
\text{\textnormal{Vol}} \left(X(\alpha)\right) &= \int_{\alpha}^{\pi} 2 \log \left| \frac{A-iV}{A+iV}\right| \: d\alpha,
\end{align*}

\noindent where 
for $A=\cot{\frac{\alpha}{2}}$, $V$ ($\text{\textnormal{Re}}(V) \leq 0$ and 
 $\text{\textnormal{Im}}(V) \geq 0$ is the largest) is a zero of the Riley-Mednykh polynomial $P=P(V,A)$ for the link $7^2_3$ given below.  
\medskip
\begin{align*}
P&= 8 V^5+8 A^2 V^4+\left(8 A^4+16 A^2-8\right) V^3+\left(4 A^6+8 A^4-12
   A^2\right) V^2 \\
   &+\left(A^8+4 A^6-2 A^4-12 A^2+1\right) V-4 A^6-8 A^4+4 A^2.
\end{align*}
\end{theorem}

The following corollary gives the hyperbolic volume of the $k$-fold strictly-cyclic covering~\cite{MM,MuV} over 
the link $7_3^2$, $M_k (X)$, for $k \geq 3$.

\begin{corollary}
The volume of $M_k (X)$ is given by the following formula

\begin{align*}
\text{\textnormal{Vol}} \left(M_k (X))\right) &=k \int_{\frac{2 \pi}{k}}^{\pi} 2 \log \left| \frac{A-iV}{A+iV}\right| \: d\alpha,
\end{align*}

\noindent where 
for $A=\cot{\frac{\alpha}{2}}$, $V$ ($\text{\textnormal{Re}}(V) \leq 0$ and 
 $\text{\textnormal{Im}}(V) \geq 0$ is the largest) is a zero of the Riley-Mednykh polynomial $P=P(V,A)$ for the link $7_3^2$.
\end{corollary}

In Section~\ref{sec:link73}, we present the fundamental group $\pi_1(X)$ of $X$ with slope $9/16$. In Section~\ref{sec:poly}, we give the defining equation of the representation variety of $\pi_1(X)$. In Section~\ref{sec:longitude}, we compute the longitude of the link $7_3^2$ using the Pythagorean theorem. And in Section~\ref{sec:proof}, we give the proof of Theorem~\ref{thm:main} using the Schl\"{a}fli formula. 
\medskip
\section{Link $7_3^2$} \label{sec:link73}

\begin{figure} 
\begin{center}
\resizebox{7cm}{!}{\includegraphics{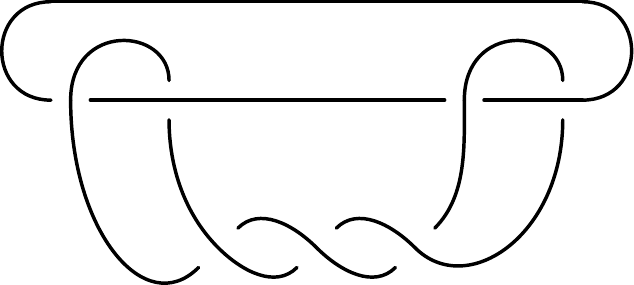}}

\caption{Link $7_3^2$ in Rolfsen's link table} \label{fig:link}
\end{center}
\end{figure}

Link $7_3^2$ is presented in Figure~\ref{fig:link}. It is the same as $W_3$ from~\cite{DMM1}.
The slope of this link is $7/16$. The link with slope $9/16$ is the mirror of the link $7_3^2$. Since the volume of the link with slope $7/16$ is the same as the volume of link with slope $9/16$, in the rest of the paper, the link with slope $9/16$ is used. 

The following fundamental group of $X$ is stated in~\cite{DMM1} with slope $7/16$.

\begin{proposition} \label{thm:fundamentalGroup}
$$\pi_1(X)=\left\langle s,t \ |\ sws^{-1}w^{-1}=1\right\rangle,$$
where $w=s^{-1}[s,t]^2[s,t^{-1}]^2$.
\end{proposition}
\section{$\left({\textnormal{PSL}}(2, \mathbb C),\mathbb{H}^3 \right) \text{ structure of } X (\alpha)$} \label{sec:poly}
Let $R=\text{Hom}(\pi_1(X), \text{SL}(2, \Cb))$. 
Given  a set of generators, $s,t$, of the fundamental group for 
$\pi_1(X)$, we define
 a set $R\left(\pi_1(X)\right) \subset \text{SL}(2, \Cb)^2 \subset \Cb^{8}$ to be the set of
 all points $(h(s),h(t))$, where $h$ is a
 representaion of $\pi_1(X)$ into $\text{SL}(2, \Cb)$. Since the defining relation of 
 $\pi_1(X)$ gives the defining equation of $R\left(\pi_1(X)\right)$~\cite{R3}, $R\left(\pi_1(X)\right)$ is an affine algebraic set in $\Cb^{8}$. 
$R\left(\pi_1(X)\right)$ is well-defined up to isomorphisms which arise from changing the set of generators. We say elements in $R$ which differ by conjugations in $\text{SL}(2, \Cb)$ are \emph{equivalent}. 
A point on the variety gives the $\left({\textnormal{PSL}}(2, \mathbb C),\mathbb{H}^3 \right) \text{ structure of } X(\alpha)$.

Let
\begin{center}
$$\begin{array}{ccccc}
h(s)= \left[\begin{array}{cc}
     \cos \frac{ \alpha}{2} & i e^{\frac{\rho}{2}} \sin \frac{ \alpha}{2}        \\
      i e^{-\frac{\rho}{2}} \sin \frac{ \alpha}{2}  &  \cos \frac{ \alpha}{2}  
                     \end{array} \right]                     
\text{,} \ \ \
h(t)=\left[\begin{array}{cc}
        \cos \frac{ \alpha}{2} & i e^{-\frac{\rho}{2}} \sin \frac{ \alpha}{2}      \\
         i e^{\frac{\rho}{2}} \sin \frac{ \alpha}{2}   &  \cos \frac{ \alpha}{2} 
                 \end{array}  \right].
\end{array}$$
\end{center}

Then $h$ becomes a representation if and only if $A=\cot{\frac{\alpha}{2}}$ and $V=\cosh{\rho}$ satisfies a polynomial equation~\cite{R3,MR2}. We call the defining polynomial 
of the algebraic set $\{(V,A)\}$ as the \emph{Riley-Mednykh polynomial} for the link $7_3^2$. Thoughout the paper, $h$ can be sometimes any representation and  sometimes the unique hyperbolic representation.

Given the fundamental group of $X$,
$$\pi_1(X)=\left \langle s,t \ |\  sws^{-1}w^{-1}=1 \right \rangle,$$
where $w=s^{-1}[s,t]^2[s,t^{-1}]^2$, let $S=h(s),\  T=h(t)$ and $W=h(w)$. Then the trace of $S$ and the trace of $T$ are both 
$2 \cos \frac{\alpha}{2}$. 

\begin{lemma}\label{lem:swn}
For $n \in \text{\textnormal{SL}}(2, \Cb)$ which satisfies $nS=S^{-1}n$, $nT=T^{-1}n$, and $n^2=-I$,
$$SWS^{-1}W^{-1}=-(SWn)^2. $$
\end{lemma}
\begin{proof}
 \begin{equation*}
\begin{split}
 (SWn)^2 & =SWnSWn=SWS^{-1}n(S^{-1}(STS^{-1}T^{-1})^2(ST^{-1}S^{-1}T)^2)n \\
              & =SWS^{-1}(S(S^{-1}T^{-1}ST)^2(S^{-1}TST^{-1})^2)n^2=-SWS^{-1}W^{-1}.
 \end{split}
\end{equation*}   
\end{proof}

From the structure of the algebraic set of $R\left(\pi_1(X)\right)$ with coordinates $h(s)$ and $h(t)$ we have the defining equation of 
$R\left(\pi_1(X)\right)$. The following theorem is stated in~\cite[Proposition 4]{DMM1} with slope $7/16$.

\begin{theorem} \label{thm:RMpolynomial}
$h$ is a representation of $\pi_1(X)$ if $V$ is a root of the following Riley-Mednykh polynomial $P=P(V,A)$ which is given below. 

\medskip

\begin{align*}
P&= 8 V^5+8 A^2 V^4+\left(8 A^4+16 A^2-8\right) V^3+\left(4 A^6+8 A^4-12
   A^2\right) V^2 \\
   &+\left(A^8+4 A^6-2 A^4-12 A^2+1\right) V-4 A^6-8 A^4+4 A^2.
\end{align*}
\end{theorem}

\begin{proof}
Note that $SWS^{-1}W^{-1}=I$, which gives the defining equations of 
$R\left(\pi_1(X)\right)$, is equivalent to $(SWn)^2=-I$ in $\text{SL}(2,\Cb)$ 
by Lemma~\ref{lem:swn} and  
$(SWn)^2=-I$ in $\text{SL}(2,\Cb)$ is equivalent to $\text{\textnormal{tr}}(SWn)=0$.

We can find  two $n$'s in $\text{\textnormal{SL}}(2, \Cb)$ which satisfies $nS=S^{-1}n$ and $n^2=-I$ by direct computations. The existence and the uniqueness of the isometry (the involution) which is represented by $n$ are shown in~\cite[p. 46]{F}. Since two $n$'s give the same element in 
$\text{\textnormal{PSL}}(2, \Cb)$, we use one of them.
Hence, we may assume
 \begin{center}
$$\begin{array}{cc}
n=\left[\begin{array}{cc}
        i & 0    \\
        0 & -i
       \end{array}  \right],
\end{array}$$
\end{center}

 \begin{center}
$$\begin{array}{ccccc}
S=\left[\begin{array}{cc}
     \cos \frac{ \alpha}{2} & i e^{\frac{\rho}{2}} \sin \frac{ \alpha}{2}        \\
      i e^{-\frac{\rho}{2}} \sin \frac{ \alpha}{2}  &  \cos \frac{ \alpha}{2}  
                     \end{array} \right],                          
\ \ \
T=\left[\begin{array}{cc}
        \cos \frac{ \alpha}{2} & i e^{-\frac{\rho}{2}} \sin \frac{ \alpha}{2}      \\
         i e^{\frac{\rho}{2}} \sin \frac{ \alpha}{2}   &  \cos \frac{ \alpha}{2} 
                 \end{array}  \right].          
\end{array}$$
\end{center}

Recall that $P$ is the defining polynomial of the algebraic set $\{(V,A)\}$ and the defining polynomial of $R\left(\pi_1(X)\right)$ corresponding to our choice of 
$h(s)$ and $h(t)$. By direct computation $P$ is a factor of $\text{\textnormal{tr}}(SWn)=-4 i \sinh{\rho} (2 V^2+ A^4+ 2 A^2-1 )  P$. As in~\cite{DMM1}, $P$ can not be $\sinh{\rho}$ or have only real roots. Also, $P$ can not have only purely imaginary roots similarly. $P$ in the theorem is the only factor of 
$\text{\textnormal{tr}}(SWn)$ which is different from $\sinh{\rho}$ and has roots which are not real or purely imaginary.
$P$ is the Riley-Mednykh polynomial.
\end{proof}

\section{Longitude}
\label{sec:longitude}
   Let $l _s= ws$ and $l_t=(t^{-1}[t,s]^2[t,s^{-1}]^2)t$. Then $l_s$ and $l_t$ are the longitudes which are null-homologus in $X$. 
  Let $L_S=h(l_s)$ and Let $L_T=h(l_t)$.
 
 \begin{lemma} \label{lem:trace}
$tr(S^{-1}L_T)=tr(S)$ and $tr(T^{-1}L_S)=tr(T)$.
\end{lemma} 

\begin{proof}
Since 

\begin{align*}
S^{-1}L_T &=S^{-1}(T^{-1}(TST^{-1}S^{-1}TST^{-1}S^{-1}\cdot TS^{-1}T^{-1}STS^{-1}T^{-1}S) T)\\
           &=(T^{-1}S^{-1}TST^{-1}S^{-1}T)(S^{-1})(T^{-1}S^{-1}TST^{-1}S^{-1}T)^{-1},
\end{align*}

$$tr(S^{-1}L_T)=tr(S^{-1})=tr(S).$$

The second statement can be obtained in a similar way.
\end{proof} 
  
\begin{definition} \label{def:longitude}
 The \emph{complex length} of the longitude $l$ ($l_s$ or $l_t$) of the link $7_3^2$ is the complex number 
 $\gamma_{\alpha}$ modulo $4 \pi \Z$ satisfying 
\begin{align*}
 \text{\textnormal{tr}}(h(l))=2 \cosh \frac{\gamma_{\alpha}}{2}.
\end{align*}
 Note that 
 $l_{\alpha}=|Re(\gamma_{\alpha})|$ is the real length of the longitude of the cone-manifold $X(\alpha)$.
\end{definition}

By sending common fixed points of $T$ and $L_{T}=h(l_t)$ to $0$ and $\infty$, we have

\begin{center}
$$\begin{array}{ccccc}
T=\left[\begin{array}{cc}
                       e^{\frac{i \alpha}{2}} & 0         \\
                        0      &  e^{-\frac{i \alpha}{2}} 
                     \end{array} \right],                          
\ \ \
L_T= \left[\begin{array}{cc}
                    e^{\frac{\gamma_{\alpha}}{2}}& 0         \\
                   0      &  e^{-\frac{\gamma_{\alpha}}{2}}
                 \end{array}  \right],
\end{array}$$
\end{center}
            
and the following normalized line matrices of $T$ (resp. $L_T$) which share the fixed points with $T$ (resp. $L_T$).

\begin{align*}
l(T) & \equiv \frac{T-T^{-1}}{2i \sinh \frac{i \alpha}{2}}  \\
        &=\frac{1}{i (e^{\frac{i \alpha}{2}}-e^{-\frac{i \alpha}{2}})}
       \left[\begin{array}{cc}
                  e^{\frac{i \alpha}{2}}-e^{-\frac{i \alpha}{2}}& 0                \\
                 0                & e^{-\frac{i \alpha}{2}}-e^{\frac{i \alpha}{2}}  
                 \end{array} \right]\\
        &=\left[\begin{array}{cc}
                  -i & 0                \\
                 0                & i
                 \end{array} \right],
 \end{align*} 

\begin{align*}
l(L_T) & \equiv \frac{L_T-L_T^{-1}}{2i \sinh \frac{\gamma_{\alpha}}{2}}  \\
        &=\frac{1}{i (e^{\frac{\gamma_{\alpha}}{2}}-e^{-\frac{\gamma_{\alpha}}{2}})}
       \left[\begin{array}{cc}
                  e^{\frac{\gamma_{\alpha}}{2}}-e^{-\frac{\gamma_{\alpha}}{2}} & 0                \\
                 0  & e^{-\frac{\gamma_{\alpha}}{2}}-e^{\frac{\gamma_{\alpha}}{2}}  
                 \end{array} \right]\\
        &=\left[\begin{array}{cc}
                 -i & 0                \\
                 0                & i
                 \end{array} \right],
 \end{align*} 
which give the orientations of axes of $T$ and $L_T$.                       

Now, we are ready to prove the following theorem which gives Theorem~\ref{thm:mpytha}.
Recall that  $\gamma_{\alpha}$ modulo $4 \pi \Z$ is the \emph{complex length} of the longitude $l_s$ or $l_t$ of  $X(\alpha)$. The following theorem is a particular case of Proposition 5 from~\cite{DMM1}.
 
\begin{theorem}(Pythagorean Theorem)~\cite{DMM1} \label{thm:pytha}
Let $X(\alpha)$ be a hyperbolic cone-manifold and let $\rho$ be the complex distance between the oriented axes $S$ and $T$. 
Then we have
$$i \cosh \rho = \cot \frac{\alpha}{2} \coth (\frac{\gamma_{\alpha}}{4}).$$
\end{theorem} 

\begin{proof}
\begin{align*}
\cosh \rho &=-\frac{\text{\textnormal{tr}}(l(S) l(T))}{2}\\
               &=-\frac{\text{\textnormal{tr}}(l(S) l(L_T))}{2} \\
               &=\frac{\text{\textnormal{tr}}((S-S^{-1}) (L_T-L_T^{-1}))}
               {8 \sinh \frac{i \alpha}{2} \sinh \frac{ \gamma_{\alpha}}{2}}\\
               &=\frac{\text{\textnormal{tr}}(SL_T-S^{-1}L_T-SL_T^{-1}+(L_TS)^{-1}))}
               {8 \sinh \frac{i \alpha}{2} \sinh \frac{ \gamma_{\alpha}}{2}}\\ 
               &=\frac{2 (\text{\textnormal{tr}}(SL_T)-\text{\textnormal{tr}}(S^{-1}L_T))}
               {8 \sinh \frac{i \alpha}{2} \sinh \frac{ \gamma_{\alpha}}{2}}\\  
               &=\frac{\text{\textnormal{tr}}(S) \text{\textnormal{tr}}(L_T) -2 \text{\textnormal{tr}}(S^{-1}L_T)}
               {4 \sinh \frac{i \alpha}{2} \sinh \frac{ \gamma_{\alpha}}{2}}\\ 
               &=\frac{\text{\textnormal{tr}}(S) \text{\textnormal{tr}}(L_T) -2 \text{\textnormal{tr}}(S)}
               {4 \sinh \frac{i \alpha}{2} \sinh \frac{ \gamma_{\alpha}}{2}}\\
               &=\frac{\text{\textnormal{tr}}(S) (\text{\textnormal{tr}}(L_T)-2)}
               {4 \sinh \frac{i \alpha}{2} \sinh \frac{ \gamma_{\alpha}}{2}}\\
               &=\frac{2 \cos \frac{\alpha}{2} (2\cosh{\frac{ \gamma_{\alpha}}{2}}-2)}
               {4 i \sin \frac{\alpha}{2} \sinh \frac{ \gamma_{\alpha}}{2}}\\
               &=-i \cot \frac{\alpha}{2} \tanh (\frac{\gamma_{\alpha}}{4}).
               \end{align*}
where the first equality comes from ~\cite[p. 68]{F}, the sixth equality comes from the Cayley-Hamilton theorem, and the seventh equality comes from Lemma~\ref{lem:trace}.
Therefore, we have
$$i \cosh \rho = \cot \frac{\alpha}{2} \coth (\frac{\gamma_{\alpha}}{4}).$$
\end{proof}

Pythagorean theorem~\ref{thm:pytha} gives the following theorem which relates the eigenvalues of $h(l)$ and $V=\cosh{\rho}$ for $A=\cot \frac{\alpha}{2}$.

\begin{theorem}\label{thm:mpytha}
Recall that $l$ is the longitude. By conjugating if necessary, we may assume $h(l)$ is upper triangular. Let $L=h(l)_{11}$. Let $A=\cot \frac{\alpha}{2}$. Then the following formulae show that there is a one to one correspondence between the the eigenvalues of $h(l)$ and $V=\cosh{\rho}$:

$$iV=A \frac{L-1}{L+1} \ and \ L= \frac{A-iV}{A+iV}.$$
\end{theorem}

\begin{proof}
By Theorem~\ref{thm:pytha},
\begin{equation*}
\begin{split}
i V &= i \cosh \rho \\
     &= \cot \frac{\alpha}{2}\tanh ( \frac{\gamma_{\alpha}}{4}) \\
     &= \cot \frac{\alpha}{2}\frac{\sinh ( \frac{\gamma_{\alpha}}{4})}
     {\cosh ( \frac{\gamma_{\alpha}}{4})} \\
     &= \cot \frac{\alpha}{2}
     \frac{e^{\frac{\gamma_{\alpha}}{4}}-e^{-\frac{\gamma_{\alpha}}{4}}}{e^{\frac{\gamma_{\alpha}}{4}}+ e^{-\frac{\gamma_{\alpha}}{4}}} \\
     &= \cot \frac{\alpha}{2} 
     \frac{e^{\frac{\gamma_{\alpha}}{2}}-1}{e^{\frac{\gamma_{\alpha}}{2}}+ 1} \\
     &=A \frac{L-1}{L + 1}.
\end{split}
\end{equation*}
If we solve the above equation, $$iV=A \frac{L-1}{L+1},$$ for $L$, we have  
$$L=\frac{A-iV}{A+iV}.$$ 
\end{proof}


\section{Proof of Theorem~\ref{thm:main}} \label{sec:proof}

According to \cite{P2,HLM1,K1,PW}, there exists an angle $\alpha_0 \in [\frac{2\pi}{3},\pi)$ such that $X(\alpha)$ is hyperbolic for $\alpha \in (0, \alpha_0)$, Euclidean for $\alpha=\alpha_0$, and spherical for $\alpha \in (\alpha_0, \pi]$ . 
Denote by $D(X(\alpha))$  the discriminant of
$P(V,A)$ over $V$. Then $\alpha_0$ is the only zero of $D(X(\alpha))$ in $[\frac{2\pi}{3},\pi)$.

From Theorem~\ref{thm:mpytha}, we have the following equality,
\begin{equation}\label{equ:absL}
\begin{split}
|L|^2 &= \left|\frac{A-iV}{A+iV}\right|^2 = \frac{|A|^2+|V|^2+2AImV}{|A|^2+|V|^2-2AImV}.
\end{split}
\end{equation}

For the volume, we choose $L$ with $|L|\geq1$ and hence we have $\text{\textnormal{Im}}(V) \geq 0$ by Equality~(\ref{equ:absL}). The component of $V$ with $\text{\textnormal{Im}}(V) \geq 0$ which becomes real at $\alpha_0$ has negative real part.
 On the geometric component  which gives the unique hyperbolic structure, we have
 the volume of a hyperbolic cone-manifold 
$X(\alpha)$ for $0 \leq \alpha < \alpha_0$:
\begin{align*}
\text{\textrm{Vol}}(X(\alpha)) &=-\int_{\alpha_0}^{\alpha} 2\left( \frac{l_{\alpha}}{2} \right) \: d\alpha \\
                        &=-\int_{\alpha_0}^{\alpha} 2 \log|L| \: d\alpha\\
                         &=-\int_{\pi}^{\alpha} 2 \log|L| \: d\alpha\\
                         &=\int^{\pi}_{\alpha} 2 \log|L| \: d\alpha\\
                         &=\int^{\pi}_{\alpha}  2 \log \left|\frac{A-iV}{A+iV}\right|\: d\alpha,                       
\end{align*}
where the first equality comes from the Schl\"{a}fli formula for cone-manifolds (Theorem 3.20 of~\cite{CHK}), the second equality comes from the fact that $l_{\alpha}=|Re(\gamma_{\alpha})|$ is the real length of the one longitude of 
$X(\alpha)$, the third equality comes from the fact that $\log|L|=0$  for $\alpha_0 < \alpha \leq \pi$ by Equality~(\ref{equ:absL}) since all $V$'s are real for $\alpha_0 < \alpha \leq \pi$, and 
$\alpha_0 \in [\frac{2 \pi}{3},\pi)$ is the zero of the discriminant $D(X(\alpha))$. Numerical calculations give us the following value  for $\alpha_0: \,  \alpha_0 \approx 2.83003.$
\medskip

\end{document}